\documentclass[12pt,thmsa]{amsproc}
\usepackage[dvips]{graphics}
\input{amssym.def}
\input{amssym.tex}

\usepackage[dvips]{graphics}

\def\bbr{{\Bbb R}}

\def\esp{{\hbox{\rm ess$\,$sup}}}

\def\rn{\bbr^n}

\def\part{\partial}
\def\intl{\int\limits}

\def\Gam{\Gamma}

\def\a{\alpha}
\def\om{\omega}

\def\del{\delta}
\def\vp{\varphi}

\def\gam{\gamma}

\def\sig{\sigma}
\def\lam{\lambda}

\def\e{\varepsilon}
\def\t{\tau}

\newtheorem{theorem}{Theorem}[section]

\theoremstyle{definition}

\theoremstyle{remark}
\newtheorem{remark}[theorem]{Remark}

\theoremstyle{corollary}

\numberwithin{equation}{section}

\newcommand{\be}{\begin{equation}}
\newcommand{\ee}{\end{equation}}

\newcommand{\bea}{\begin{eqnarray}}
\newcommand{\eea}{\end{eqnarray}}
\newcommand{\Bea}{\begin{eqnarray*}}
\newcommand{\Eea}{\end{eqnarray*}}

\def\sideremark#1{\ifvmode\leavevmode\fi\vadjust{\vbox to0pt{\vss
 \hbox to 0pt{\hskip\hsize\hskip1em
\vbox{\hsize2cm\tiny\raggedright\pretolerance10000
 \noindent #1\hfill}\hss}\vbox to8pt{\vfil}\vss}}}%

\begin{document}

\title[Semyanistyi fractional integrals]
{Weighted norm estimates for the Semyanistyi fractional integrals and  Radon transforms}

\author{B. Rubin}
\address{
Department of Mathematics, Louisiana State University, Baton Rouge,
LA, 70803 USA}

\email{borisr@math.lsu.edu}



\dedicatory{Dedicated to Professor Sigurdur Helgason
on  his 85th birthday}


\begin{abstract}
 Semyanistyi's fractional integrals have come to analysis from  integral geometry. They take functions on $\bbr^n$ to functions on hyperplanes,
 commute with rotations, and have a nice behavior with respect to dilations. We obtain sharp  inequalities for these integrals and the corresponding  Radon  transforms
 acting on  $L^p$ spaces with a radial power weight. The operator norms are explicitly  evaluated. Similar results are obtained for fractional integrals associated to $k$-plane transforms for any $1\le k<n$.
\end{abstract}

\maketitle

\section{Introduction}

\setcounter{equation}{0}

The  Radon transforms play the central role in  integral geometry and have numerous applications \cite{Dea, He, Mar, Na}.
 In 1960 V.I. Semyanistyi \cite{Se1}
came up with an  idea to
 regard the  Radon transform on $\bbr^n$
 as a member of suitable analytic family of fractional integrals $R^\a f$,
  so  that for sufficiently good $f$,
 \be\label {nton} R^\a f\big |_{\a=0}=R f.\ee
 This idea paves the way to  a variety
of inversion formulas for the Radon transform and has proved to be useful in subsequent
developments; see, e.g., \cite{OR} and references therein.

We recall basic definitions. Let $\Pi_n$ be the set of all hyperplanes in $\bbr^n$. The  Radon transform takes a function $f (x)$ on $\bbr^n$ to a function $(Rf)(\t)= \int_\t f$ on $\Pi_n$. The corresponding Semyanistyi's fractional integrals  are defined by
   \be\label{tong2} (R^\a
f)(\t)=\frac{1}{\gamma_1(\a)}\intl_{\bbr^n} f(x)\,[{\rm
dist}(x,\t)]^{\a -1}\, dx,
\ee
\be\label{tong22} \gamma_1(\a)=
  2^\a\pi^{1/2}\Gamma(\a/2)/\Gamma((1-\a)/2),\quad \a >0; \; \a \neq 1,3, \ldots,\ee
 where ${\rm dist}(x,\t)$ is the
Euclidean distance between the point $x$ and the hyperplane $\t$.
The  normalizing coefficient in   (\ref{tong2}) is inherited from  one-dimensional Riesz potentials \cite{La, Ru96}.
More general integrals, when $\t$ is a plane of arbitrary dimension $1\!\le\! k\!\le \!n\!-\!1$, were introduced in \cite{Ru04b}.

 Mapping properties of  Radon transforms   were studied in numerous publications
 from different points of view; see, e.g.,  \cite{Cal, Chr84, Dr87, DNO, F, KR10, LT, Na, OS, Q, So,
  Str}, to mention a few. In Section 2 we suggest an alternative approach which works well in weighted $L^p$ spaces with a radial power weight,  yields sharp  estimates and explicit formulas for the operator norms, and  extends to fractional integrals (\ref{tong2}). Similar results are obtained in Section 3 for more general operators associated to the $k$-plane transform for any $1\le k\le n-1$.  Main results of the paper are presented by Theorems \ref{o9iu},  \ref{o9iur}, \ref{o9iuk},  and \ref{o9iurk}.

  Our approach  was
 inspired by a series of works on
operators with homogeneous kernels; see, e.g.,  \cite {Kar, Sa, Str69, Wa}. A  common point for this  class of operators and the afore-mentioned operators of integral geometry is  a nice behavior with respect to rotations and dilations.
The same approach is applicable to the dual Radon transforms and the corresponding dual Semyanistyi integrals; see, e.g., \cite {Ru12}, where a different technique has been used.

\section {Fractional integrals associated to the Radon transform}

\subsection {Preliminaries}  In the following $ \; \sigma_{n-1}  = 2 \pi^{n/2}/
 \Gamma(n/2)$ is the area of the unit sphere $S^{n-1}$ in $\rn; \;$ $d\sig(\eta)$ stands for the surface element of $S^{n-1}$;
  $e_1, \ldots, e_n$ are coordinate unit
 vectors; $O(n)$ is the  group of orthogonal  transformations of $\bbr^n$ endowed with the
 invariant probability measure. The $L^p$ spaces, $1\le p\le \infty$, are defined in a usual way,
 $1/p+1/p'=1$. We recall that $\Pi_n$ denotes the set of all hyperplanes $\t$ in $\bbr^n$. Every  $\t\in\Pi_n$ can be  parametrized as
 \be\label{hppl} \t(\theta, t)=\{x\in \bbr^n: x \cdot \theta=t\}, \qquad (\theta, t)\in
 S^{n-1}\times \bbr. \ee
 Clearly, \be\label{00ma} \t (\theta, t)=\t(-\theta, -t).\ee
  We set
  \[
L^p_\mu(\bbr^n)=\{ f: ||f||_{p,\mu}\equiv |||x|^\mu f||_{L^p(\bbr^n)}<\infty \},\]
 \[
L^p_\nu (S^{n-1}\times \bbr)=\{ \vp: ||\vp||^\sim_{p,\nu}\equiv |||t|^\nu \vp||_{L^p(S^{n-1}\times \bbr)}<\infty \},\]
where $\mu$ and $\nu$ are real numbers.
 Passing to polar coordinates, we have
\be\label{00ol}
||f||_{p,\mu}\!=\! \left (\!\sigma_{n-1}\intl_0^\infty\! r^{n-1+\mu p}\,dr\!\!\intl_{O(n)} \!\!|f (r\gam e_1)
|^p\,d\gam\right )^{1/p}\!\!\!\!, \quad 1\le p< \infty,\ee
\be\label{00olb} ||f||_{\infty,\mu}=\underset{r, \gam}{\esp} \, r^\mu\,|f (r\gam e_1)|.\ee
Similarly,
\be\label{zzxs}
||\vp||^\sim_{p,\nu}\!=\! \left (\!\sigma_{n-1}\intl_{-\infty}^\infty \!\!|t|^{\nu p}\,dt\!\intl_{O(n)} \!|\vp (\gam e_1, t)
|^p\,d\gam\right )^{1/p}\!\!\!\!, \quad 1\le p< \infty,\ee

\be\label{zzxsb}||\vp||^\sim_{\infty,\nu}=\underset{t, \gam}{\esp} \, |t|^\nu\,|\vp (\gam e_1, t)|.\ee
We write the Radon transform in terms of the parametrization  (\ref {hppl}) as
\be\label{rtra1} (Rf)(\t)\equiv (Rf) (\theta, t)=\intl_{\theta^{\perp}} f(t\theta +
u) \,d_\theta u,\ee
 where $\theta^\perp\!=\!\{x: \, x \cdot \theta\!=\!0\}$,
$d_\theta u$ denotes the Lebesgue measure on $\theta^{\perp}$. Similarly,
\be\label{tong} (R^\a f)(\t)\equiv (R^\a
f)(\theta, t)=\frac{1}{\gamma_1 (\a)}\intl_{\bbr^n}
f(x)|t-x\cdot\theta|^{\a -1} \, dx. \ee

\begin{theorem}\label {o9iu} Let  $1\le p\le \infty$, $1/p+1/p'=1$,  $\a >0$. Suppose
\be\label{op98} \nu=\mu -\a-(n-1)/p',\ee
\be\label{op981} \a-1+n/p'<\mu<n/p'.\ee
Then
\be\label{op982}
||R^\a f||^\sim_{p,\nu}\le c_\a \,||f||_{p,\mu},\ee
where
\[
c_\a=||R^\a||=\frac{\displaystyle  {2^{1/p -\a}\, \pi^{(n-1)/2}\,\Gam \left(\frac{n/p' -\mu}{2}\right )\, \Gam \left(\frac{1-\a+\mu -n/p'}{2}\right)}}{\displaystyle{\Gam \left(\frac{\mu +n/p}{2}\right)\,\Gam \left(\frac{\a+n/p' -\mu}{2}\right)}}\,
.\]
\end{theorem}
\begin{remark} The necessity of  (\ref{op98}) can be proved using the standard scaling argument; cf. \cite[p. 118]{St}. Let $f_\lam (x)=f(\lam x)$, $\lam>0$. Then
\[ ||f_\lam ||_{p,\mu}=\lam^{-\mu-n/p}\,||f||_{p,\mu}, \qquad ||R^\a f_\lam ||^\sim_{p,\nu}=\lam^{-\nu-n-\a+1/p'}\,||R^\a f||^\sim_{p,\nu}.\]
Hence, whenever $||R^\a f||^\sim_{p,\nu}\le c \,||f||_{p,\mu}$ with $c$ independent of $f$, we get
$\nu=\mu -\a-(n-1)/p'$.
 As we shall see below, the condition  (\ref{op981}) is also best possible.
\end{remark}

\subsection {Proof of Theorem \ref{o9iu}}${}$\hfill

\subsubsection {Step 1}
 Let us prove (\ref{op982}).  For $t>0$, changing variables $\theta=\gam e_1$, $\gam \in O(n)$, and $x=t\gam y$, we get
\be\label{tongyi} (R^\a f)(\gam e_1, t)=\frac{t^{\a+n-1}}{\gamma_1 (\a)}\intl_{\bbr^n}
f(t
\gam y)\, |1-y_1|^{\a -1} \, dy. \ee
If $1\le p< \infty$, then, by Minkowski's inequality, using (\ref {zzxs}) and (\ref{00ol}), we obtain that the norm  $||R^\a f||^\sim_{p,\nu}$ does not exceed the following:
\bea
&&
\frac{2^{1/p}}{\gamma_1 (\a)}\intl_{\bbr^n}|1-y_1|^{\a -1}
\left (\sigma_{n-1}\intl_0^\infty \intl_{O(n)} t^{(\a+n-1+\nu) p}\,|f (t\gam y)|^p \,d\gam dt \right )^{1/p} \!dy\nonumber\\
&&\label{ppza}\qquad =c_\a \, ||f||_{p,\mu},\qquad  c_\a=\frac{2^{1/p}}{\gamma_1 (\a)}\intl_{\bbr^n}|1-y_1|^{\a -1} |y|^{-\mu-n/p}\, dy.\eea
If $p=\infty$, then (\ref {zzxsb}) and (\ref{00olb}) give a similar estimate
$||R^\a f||^\sim_{\infty,\nu}\le c_\a \, ||f||_{\infty,\mu}$ in which $c_\a$ has the same form  with $1/p=n/p=0$.

To evaluate $c_\a$, denoting $\lam=\mu +n/p$, we have
\[
 c_\a=\frac{2^{1/p}}{\gamma (\a)}\intl_{-\infty}^\infty |1-y_1|^{\a -1}\,dy_1
  \intl_{\bbr^{n-1}}    (|y'|^2 +y_1^2)^{-\lam/2}\, dy'=c_1 \,c_2,\]
\[c_1= 2^{1/p}\intl_{\bbr^{n-1}}    (1+|z|^2)^{-\lam/2}\, dz=\frac{2^{1/p}\,\pi^{(n-1)/2}\,\Gam ((\lam +1-n)/2)}
  {\Gam (\lam/2)},\]
\[
  c_2=\frac{1}{\gamma_1 (\a)}\intl_{-\infty}^\infty |1-y_1|^{\a -1}|y_1|^{n-\lam-1}\,dy_1=
  \frac{\gamma_{1}(n-\lam)}{\gamma_{1}(\a+n-\lam)};\]
see, e.g., \cite{La}, where  convolutions of Riesz kernels are considered in any dimensions. Combining these formulas with (\ref{tong22}), we obtain the desired expression. We observe that the  integral in (\ref{ppza}) is finite  if and only if
$\a-1+n/p'<\mu<n/p'$, which is (\ref{op981}).

\subsubsection {Step 2}\label{00aq} Let us show that $c_\a=||R^\a||$.  By Step 1, $||R^\a||\le c_\a$. Thus, it remains to prove that $||R^\a||\ge c_\a$.
Since the operator $R^\a: L^p_\mu (\bbr^n) \to L^p_\nu (S^{n-1} \times \bbr)$ is bounded, then for any $f\in L^p_\mu (\bbr^n)$ and $\vp \in L^{p'}_{-\nu} (S^{n-1} \times \bbr)$ by H\"older's inequality we have
\be\label{lloi}
I=\Bigg |\intl_{S^{n-1} \times \bbr}(R^\a f)(\theta, t)\, \vp (\theta, t)\, dt d\theta \Bigg |\le ||R^\a||\,  ||f||_{p,\mu}\, ||\vp||^\sim_{p', -\nu}.\ee
Suppose $ f(x)\equiv f_0(|x|)\ge 0$ and  $\vp (\theta, t)\equiv \vp_0 (|t|)\ge 0$. Then
\bea
I\!&=&\!\frac{2\sigma_{n-1}}{\gamma_1 (\a)}\intl_0^\infty \vp_0 (t)\, dt \intl_{\bbr^n}
f_0(|x|)\, |t-x\cdot e_1|^{\a -1} \, dx\nonumber\\
&=&\!\frac{2\sigma_{n-1}}{\gamma_1 (\a)}\intl_0^\infty \vp_0 (t)\, dt \intl_0^\infty r^{n-1} f_0(r)\,dr \intl_{S^{n-1} } |t-r\eta_1|^{\a -1} \, d\sig (\eta)\nonumber\\
&=&\!\frac{2\sigma_{n-1}}{\gamma_1 (\a)} \!\intl_{S^{n-1} } \!\! d\sig (\eta)\intl_0^\infty |1\!-\!s\eta_1|^{\a -1} s^{n-1}\,ds\intl_0^\infty\! \vp_0 (t) f_0(ts)\, t^{n+\a-1}\, dt.\nonumber\eea
Let $1< p<\infty$. Then
\bea
\label {844f}||f||_{p,\mu}&=&\Bigg (\sigma_{n-1} \intl_0^\infty r^{n-1+\mu p} f_0^p(r)\,dr\Bigg )^{1/p}, \\
||\vp||^\sim_{p', -\nu}&=&\Bigg (2\sigma_{n-1} \intl_0^\infty t^{-\nu p'} \vp_0^{p'}(t)\,dt\Bigg )^{1/p'}.\nonumber\eea
Choose $\vp_0(t)=t^{\mu p -\a} f_0^{p-1}(t)$ so that $||\vp||^\sim_{p', -\nu}=2^{1/p'}||f||_{p,\mu}^{p-1}$. Then  (\ref{lloi}) yields
\bea
&&\frac{2^{1/p}\sigma_{n-1}}{\gamma_1 (\a)}\intl_{S^{n-1} } \!\! d\sig (\eta)\intl_0^\infty |1\!-\!s\eta_1|^{\a -1} s^{n-1}\,ds\nonumber\\
\label{lloiz}&&\times \intl_0^\infty\! t^{\mu p +n-1} f_0^{p-1}(t) f_0(ts)\, dt\le ||R^\a||\,  ||f||^p_{p,\mu}.\eea
Finally we set $f_0(t)=0$ if $t<1$ and  $f_0(t)=t^{-\mu-n/ p -\e}$, $\e>0$, if $t>1$. Then $||f||^p_{p,\mu}=\sigma_{n-1}/ \e p$ and
 (\ref{lloiz}) becomes
\bea
||R^\a||&\ge& \frac{2^{1/p}}{\gamma_1 (\a)}\intl_{S^{n-1} } \!\! d\sig (\eta)\intl_0^\infty |1\!-\!s\eta_1|^{\a -1} s^{n/p' - \mu -1-\e}\left \{\begin{array} {ll}s^{\e p}, & s<1 \\
1, & s>1 \\
\end{array}\right \}\, ds\nonumber\\
&=& \frac{2^{1/p}}{\gamma_1 (\a)}\intl_{\bbr^n}|1-y_1|^{\a -1} |y|^{-\mu-n/p-\e}\left \{\begin{array} {ll} |y|^{\e p}, & |y|<1 \\
1, & |y|>1 \\
\end{array}\right \}\, dy;\nonumber\eea
cf. (\ref {ppza}). Passing to the limit as $\e \to 0$, we obtain  $||R^\a||\ge c_\a$.

If $p=1$, then $\nu=\mu-\a$. We choose $\vp_0 (t)=t^{\mu-\a}$ and proceed as above.
If $p=\infty$,   we choose $f_0 (r)=r^{-\mu}$. Then   $ ||f||_{\infty,\mu}=1$,
\[
 I=\frac{2\sigma_{n-1}}{\gamma_1 (\a)} \!\intl_{S^{n-1} } \!\! d\sig (\eta)\intl_0^\infty |1\!-\!s\eta_1|^{\a -1} s^{n-\mu-1}\,ds\intl_0^\infty\! \vp_0 (t) \, t^{n+\a-\mu-1}\, dt,\]
 and $I\le ||R^\a||\,   ||\vp||^\sim_{1, -\nu}$.
We set $\vp_0 (t)=0$  if $t<1$ and $\vp_0 (t)=t^{-\del}$ if $t>1$, where $\del$ is big enough. This gives
  \[ \frac{1}{\gamma_1 (\a)}\intl_{\bbr^n}|1-y_1|^{\a -1} |y|^{-\mu}\, dy\le ||R^\a||,\]
  or $c_\a \le ||R^\a||$, as desired. \hfill $\square$

\subsection{The case $\a=0$}
This case corresponds to the Radon transform (\ref{rtra1}) and requires  independent consideration.
\begin{theorem}\label {o9iur} Let  $1\le p\le \infty$, $1/p+1/p'=1$,
\be\label{op98r} \nu=\mu -(n-1)/p',\qquad  \mu>n/p'-1.\ee
Then $
||R f||^\sim_{p,\nu}\le c \,||f||_{p,\mu}$, where
\be\label{op983r}
c=||R||=\frac{\displaystyle  {2^{1/p}\, \pi^{(n-1)/2}\, \Gam \left(\frac{1+\mu -n/p'}{2}\right)}}{\displaystyle{\Gam \left(\frac{\mu +n/p}{2}\right)}}\,
.\ee
\end{theorem}
\begin{proof} The necessity of  (\ref{op98r}) can be checked  as in Theorem  \ref{o9iu}. To prove the norm inequality,  setting $\theta=\gam e_1$, $\gam \in O(n)$,  $t>0$, we have
\[(Rf) (\theta, t)=\intl_{\theta^{\perp}} f(t\theta +
u) \,d_\theta u=t^{n-1}\intl_{\bbr^{n-1}} f(t\gam (e_1 +z))\, dz.\]
Hence,
\bea
&&||R f||^\sim_{p,\nu}=\Bigg (2\sigma_{n-1}\intl_0^\infty \intl_{O(n)} \Bigg |\,\intl_{\bbr^{n-1}} t^{(n-1+\nu) p}\,|f (t\gam (e_1 +z))|^p \,dz  \,\Bigg |\,d\gam dt \Bigg )^{1/p}\nonumber\\
&&\le (2\sigma_{n-1})^{1/p}\intl_{\bbr^{n-1}}\Bigg (\intl_0^\infty \intl_{O(n)}|f (t\gam (e_1 +z))|^p \,t^{(n-1+\nu) p}\,d\gam dt \Bigg )^{1/p}dz\nonumber\\
&& \qquad \qquad\qquad \mbox{\rm (set  $t|e_1 +z|=t(1+|z|^2)^{1/2}=s$)}\nonumber\\
&& =c\,\Bigg (\intl_0^\infty \intl_{S^{n-1}}|f (s\eta)|^p \,s^{(n-1+\nu) p}\,d\sigma (\eta) ds \Bigg )^{1/p}=c\, ||f||_{p,\mu},\nonumber\eea
where
\[c= 2^{1/p}\,\intl_{\bbr^{n-1}}   \frac{dz}{ (1+|z|^2)^{(\nu+n-1/p')/2}}= 2^{1/p}\,\intl_{\bbr^{n-1}}   \frac{dz}{ (1+|z|^2)^{(\mu+n/p)/2}}\,.\]
The last integral gives an expression in (\ref{op983r}).
The proof of the equality $c=||R||$  mimics that in  Section \ref{00aq}.
\end{proof}

\section {Fractional integrals associated to the $k$-plane transforms}

We denote by $\Pi_{n,k}$  the set of all nonoriented $k$-dimensional  planes  in $\bbr^n$, $1\le k\le n-1$.
To  parameterize  such planes and define the  corresponding analogues of $R$ and $R^\a$, we introduce the Stiefel manifold
$V_{n, n-k} \sim O(n)/O(k)$    of $n\times (n-k)$ real matrices, the columns of which
 are mutually orthogonal unit $n$-vectors. For $v\in V_{n, n-k}$,  $dv$ stands for the  left $O(n)$-invariant probability measure on $V_{n, n-k}$ which is also right $O(n-k)$-invariant.
  Every  plane $\t\in \Pi_{n,k}$ can be  parameterized by
 \be\label{hpplk} \t(v, t)=\{x\in \bbr^n: v^T x =t\}, \qquad (v, t)\in
 V_{n, n-k}\times \bbr^{n-k}, \ee
 where $v^T$ stands for the transpose of the matrix $v$.
 The case $k=n-1$  agrees with  (\ref{hppl}).
  Clearly, \be\label{00mak} \t(v, t)=\t(v\om^T, \om t) \quad \forall\,  \om \in O(n-k).\ee
 This equality is a  substitute for  (\ref{00ma}) for all  $1\le k\le n-1$.

 The $k$-plane transform takes a function $f$ on $\bbr^n$ to a function $(R_k f)(\t)=\int_\t f$ on $\Pi_{n,k}$.   In terms of  the parametrization (\ref {hpplk}) it has the form
\be\label{rtra1k} (R_k f) (v, t)=\intl_{v^{\perp}} f(vt +
u) \,d_v u,\ee
 where $v^{\perp}$ denotes the $k$-dimensional linear subspace orthogonal to   $v$ and
$d_v u$ stands the usual Lebesgue measure on $v^{\perp}$. Similarly,
\be\label{tongk} (R_k^\a
f)(v, t)=\frac{1}{\gamma_{n-k} (\a)}\intl_{\bbr^n}
f(x)\,|v^T x -t|^{\a +k-n} \, dx, \ee
where $|v^T x -t|$ denotes for the Euclidean norm of $v^T x -t$ in $\bbr^{n-k}$,
\be\label {012a} \gamma_{n-k}(\a)=
  \frac{2^\a\pi^{(n-k)/2}\Gamma(\a/2)}{\Gamma((n-k-\a)/2)}.
  \ee
Here $1/\gamma_{n-k} (\a)$ is the same normalizing coefficient as in the Riesz potential on $\bbr^{n-k}$ \cite{La, Ru96}.

\begin{remark} One can regard $\Pi_{n,k}$ as a fiber bundle over the Grassmann manifold of all $k$-dimensional linear subspaces of $\bbr^n$. The corresponding parametrization of $k$-planes is  different from (\ref{hpplk}); cf. \cite{Ru04b, Ru12}. In the present article we prefer to work with parametrization (\ref{hpplk}) because it reduces to   (\ref{hppl}) when $k=n-1$.
\end{remark}

Let
\[
L^p_\nu ( V_{n, n-k}\times \bbr^{n-k})=\{ \vp: ||\vp||^\sim_{p,\nu}\equiv
 ||\,|t|^\nu \vp||_{L^p(V_{n, n-k}\times \bbr^{n-k})}<\infty \},\]
 where the $L^p$ norm on $V_{n, n-k}\times \bbr^{n-k}$ is taken with respect to the measure $dvdt$;
$$v_0=
\left[\begin{array} {cc} I_{n-k} \\ 0
 \end{array} \right]\in V_{n, n-k},$$  $I_{n-k}$ is the identity $(n-k)\times (n-k)$ matrix. Then
  \be\label{zzxsk}
||\vp||^\sim_{p,\nu}\!=\! \left (\!\sigma_{n-k-1}\!\intl_0^\infty \!r^{\nu p+n-k-1}\,dr\!\!\intl_{O(n)}\!\! d\gam \!\!\! \intl_{O(n-k)} \!\!\!\!
|\vp (\gam v_0, r\om e_1)
|^p\,d\om\! \right )^{1/p}\!\!\!\!, \ee
 if $1\le p< \infty$,  and
\be\label{zzxsbk}||\vp||^\sim_{\infty,\nu}=\underset{r, \gam, \om}{\esp} \, |r|^\nu\,|\vp (\gam v_0, r\om e_1)
|.\ee

\begin{theorem}\label {o9iuk} Let  $1\le p\le \infty$, $1/p+1/p'=1$,  $\a >0$. Suppose
\be\label{op98k} \nu=\mu -\a-k/p',\ee
\be\label{op981k} \a+k-n/p<\mu<n/p'.\ee
Then $||R_k^\a f||^\sim_{p,\nu}\le c_{\a,k} \,||f||_{p,\mu},$
where
\[
c_{\a,k}\!=\!||R_k^\a||\!=\!2^{-\a}\pi^{k/2}\left (\frac{\sigma_{n-k-1}}{\sigma_{n-1}}\right )^{1/p}
\,\frac{\displaystyle  { \Gam \left(\frac{n/p'\! -\!\mu}{2}\right )\, \Gam \left(\frac{\mu \!+\!n/p \! -\!k\!-\!\a}{2}\right)}}{\displaystyle{\Gam \left(\frac{\mu\! +\!n/p}{2}\right)\,\Gam \left(\frac{n/p'\! -\!\mu\!+\!a}{2}\right)}}\,
.\]
\end{theorem}

\subsection {Proof of Theorem \ref{o9iuk}}${}$\hfill

\subsubsection {Step 1}\label{00aqn} The necessity of  (\ref{op98k}) is a consequence of  the equalities
\[ ||f_\lam ||_{p,\mu}=\lam^{-\mu-n/p}\,||f||_{p,\mu}, \qquad ||R_k^\a f_\lam ||^\sim_{p,\nu}=\lam^{-\nu-\a-k/p'-n/p}\,||R_k^\a f||^\sim_{p,\nu},\]
where $f_\lam (x)=f(\lam x)$, $\lam>0$.
Let $\gam \in O(n)$ and $\om \in O(n-k)$ be such that $v=\gam v_0$, $t=|t|\om e_1$.
Changing variables in (\ref{tongk}) by setting
\[ x=|t|\gam \tilde\om y, \qquad \tilde \om=\left[\begin{array} {cc} \om & 0 \\ 0 & I_k
 \end{array} \right],\]
 we obtain
\be\label{tongyik} (R_k^\a f)(\gam v_0, |t| \om e_1)=\frac{|t|^{\a+k}}{\gamma_{n-k} (\a)}\intl_{\bbr^n}
f(|t|
\gam \tilde\om y)\, |v_0^T y-e_1|^{\a +k-n} \, dy. \ee
Suppose first $1\le p< \infty$. Then, by Minkowski's inequality, owing to (\ref{zzxsk}) and (\ref{op98k}),  the norm  $||R_k^\a f||^\sim_{p,\nu}$ does not exceed the following:
\bea
&&
\left (\!\sigma_{n-k-1}\!\intl_0^\infty \intl_{O(n)} \intl_{O(n-k)}\!r^{\nu p+n-k-1} | (R_k^\a f)(\gam v_0, r \om e_1)|^p \,d\om d\gam dr\right )^{1/p}\nonumber\\
&&\le
\frac{\sigma_{n-k-1}^{1/p}}{\gamma_{n-k} (\a)}\intl_{\bbr^n}
 |v_0^T y-e_1|^{\a +k-n} \nonumber\\
 &&\times\left (\intl_0^\infty \intl_{O(n)} \intl_{O(n-k)} |f(r\gam \tilde\om y)|^p
r^{
\mu p+n-1}  \,d\om d\gam dr\right )^{1/p}dy=c_{\a,k} \, ||f||_{p,\mu},\nonumber\eea
\be\label{654d}
c_{\a,k}= \left (\frac{\sigma_{n-k-1}}{\sigma_{n-1}}\right )^{1/p}
\frac{1}{\gamma_{n-k} (\a)} \intl_{\bbr^n}
 |v_0^T y-e_1|^{\a +k-n} |y|^{-\mu-n/p}\, dy.\ee
If $p=\infty$ we similarly have
$||R_k^\a f||^\sim_{\infty,\nu}\le c_{\a,k}\, ||f||_{\infty,\mu}$, where  $c_{\a,k}$ has the same form as above with $1/p=n/p=0$.

To compute $c_{\a,k}$, we set $\lam=\mu+n/p$. Then
\[
c_{\a,k}\!=\!\frac{\sigma_{n-k-1}^{1/p}}{\sigma_{n-1}^{1/p}\, \gamma_{n-k} (\a)} \intl_{\bbr^{n-k}}\!\!
|y'-e_1|^{\a +k-n}\, dy'\intl_{\bbr^{k}}(|y'|^2 +|y''|^2)^{-\lam/2}\, dy''\!=\!c_{1}\,c_{2},\]
 where
\[c_{1}= \left (\frac{\sigma_{n-k-1}}{\sigma_{n-1}}\right )^{1/p} \intl_{\bbr^{k}}    (1+|z|^2)^{-\lam/2}\, dz=\left (\frac{\sigma_{n-k-1}}{\sigma_{n-1}}\right )^{1/p}\frac{\pi^{k/2}\,\Gam ((\lam -k)/2)}
  {\Gam (\lam/2)},\]
\[
  c_{2}=\frac{1}{\gamma_{n-k} (\a)} \intl_{\bbr^{n-k}}
|y'-e_1|^{\a +k-n}\,|y'|^{k-\lam}\, dy'=
   \frac{\gamma_{n-k}(n-\lam)}{\gamma_{n-k}(n-\lam+\a)};\]
see also (\ref{012a}). It remains to put these formulas together  and make obvious simplifications.
Note that the repeated integral in the expression  for $c_{\a,k}$ is finite  if and only if
$\a+k-n/p<\mu<n/p'$, which is (\ref{op981k}). Thus $||R_k^\a||\le c_{\a,k}$.

\subsubsection {Step 2} To prove that $||R_k^\a||\ge c_{\a,k}$ we follow the reasoning from Section \ref{00aq}. For any $f\in L^p_\mu (\bbr^n)$ and $\vp \in L^{p'}_{-\nu} ( V_{n, n-k}\times \bbr^{n-k})$,
\[
I=\Bigg |\intl_{V_{n, n-k}\times \bbr^{n-k}}(R_k^\a f)(v, t)\, \vp (v, t)\, dt dv \Bigg |\le ||R_k^\a||\,  ||f||_{p,\mu}\, ||\vp||^\sim_{p', -\nu}.\]
Let
 $ f(x)\equiv f_0(|x|)\ge 0$,  $\vp (v, t)\equiv \vp_0 (|t|)\ge 0$. Then
\bea
&&\!\!\!\!\!I\!=\!\frac{1}{\gamma_{n-k} (\a)}\intl_{\bbr^{n-k}} \vp_0 (|t|)\,|t|^{\a+k}\, dt \intl_{\bbr^n}
f_0(|t||y|)\,|v_0^T y-e_1|^{\a +k-n} \, dy \nonumber\\
&&\!\!\!\!\!=\!\frac{\sigma_{n-k-1}}{\gamma_{n-k} (\a)} \!\intl_{S^{n-1} }\! \!\! d\sig (\eta)\intl_0^\infty |sv_0^T \eta\!-\!e_1|^{\a +k-n} s^{n-1}\,ds \intl_0^\infty\! \vp_0 (r) f_0(rs)\, r^{n+\a-1}\, dr.\nonumber\eea
Suppose $1< p<\infty$. Then  $||f||_{p,\mu}$ can be computed by (\ref{844f}) and
\[
||\vp||^\sim_{p', -\nu}=\left (\!\sigma_{n-k-1}\!\intl_0^\infty \!r^{-\nu p'+n-k-1}\,\vp_0^{p'}\,dr \right )^{1/p}.\]
cf. (\ref{zzxsk}). Choose $\vp_0(r)=r^{\mu p -\a} f_0^{p-1}(r)$. Then
$$
||\vp||^\sim_{p', -\nu}=\left (\frac{\sigma_{n-k-1}}{\sigma_{n-1}}\right )^{1/p'}||f||_{p,\mu}^{p-1}$$ and we have
\bea
&&\!\!\!\!\!\!\frac{\sigma_{n-k-1}}{\gamma_{n-k} (\a)} \!\intl_{S^{n-1} } \!\! d\sig (\eta)\intl_0^\infty |sv_0^T \eta-e_1|^{\a +k-n} s^{n-1}\,ds \intl_0^\infty\! f_0^{p-1}(r) f_0(rs)\, r^{\mu p +n-1}\, dr\nonumber\\
&& \qquad \qquad {}\le \left (\frac{\sigma_{n-k-1}}{\sigma_{n-1}}\right )^{1/p'} ||R_k^\a||\,  ||f||^p_{p,\mu}.\label {kkibf}\eea
Setting $f_0(r)=0$ if $r<1$ and  $f_0(r)=r^{-\mu-n/ p -\e}$, $\e>0$, if $r>1$, we obtain $||f||^p_{p,\mu}=\sigma_{n-1}/ \e p$, and therefore,
\bea
||R_k^\a||\!&\ge& \!\! \left (\frac{\sigma_{n-k-1}}{\sigma_{n-1}}\right )^{1/p}
\frac{1}{\gamma_{n-k} (\a)} \intl_{S^{n-1} }  d\sig (\eta) \nonumber\\
&\times& \!\!\intl_0^\infty |sv_0^T \eta-e_1|^{\a +k-n}
s^{n-1- \mu -n/p -\e}
\left \{\begin{array} {ll}s^{\e p}, & s<1 \\
1, & s>1 \\
\end{array}\right \}\, ds\nonumber\\
&=& \!\! \left (\frac{\sigma_{n-k-1}}{\sigma_{n-1}}\right )^{1/p}\!\!\!
\frac{1}{\gamma_{n-k} (\a)} \intl_{\bbr^{n} }
\frac{ |v_0^T y\!-\!e_1|^{\a +k-n}}{
|y|^{\mu +n/p +\e}}
\left \{\begin{array} {ll} |y|^{\e p}, & |y|<1 \\
1, & |y|>1 \\
\end{array}\right \} dy;\nonumber\eea
cf. (\ref{654d})
Passing to the limit as $\e \to 0$, we obtain  $||R_k^\a||\ge c_{\a,k}$.
The cases $p=1$ and $p=\infty$ are treated as in Section \ref{00aq}.

\subsection{Weighted norm estimates for the $k$-plane transform}
The following statement deals with the $k$-plane transform  (\ref{rtra1k}) and formally corresponds to  $\a=0$ in Theorem \ref {o9iuk}.
\begin{theorem}\label {o9iurk} Let  $1\le p\le \infty$, $1/p+1/p'=1$. Suppose that
\be\label{op98rk} \nu=\mu -k/p',\qquad  \mu> k-n/p.\ee
Then $
||R_k f||^\sim_{p,\nu}\le c_k \,||f||_{p,\mu}$, where
\be\label{op983rk}
c_k=||R_k||=\!\pi^{k/2}\left (\frac{\sigma_{n-k-1}}{\sigma_{n-1}}\right )^{1/p}
\,\frac{\displaystyle  {\Gam \left(\frac{\mu \!+\!n/p \! -\!k}{2}\right)}}{\displaystyle{\Gam \left(\frac{\mu\! +\!n/p}{2}\right)}}\,
.\ee
\end{theorem}
\begin{proof} As before, the conditions  (\ref{op98rk}) are sharp. To prove the norm inequality,  as in Section \ref {00aqn}
we set $v=\gam v_0$, $t=r\om e_1$,  $ \gam \in O(n)$, $ \om \in O(n-k)$, $r>0$. This gives
 \[ (R_k f)(\gam v_0, r \om e_1)=r^{k}\intl_{\bbr^k}
f(r\gam \tilde\om  (e_1+z)) \, dz, \qquad  \tilde \om=\left[\begin{array} {cc} \om & 0 \\ 0 & I_k
 \end{array} \right].\]
If $1\le p< \infty$, then combining (\ref{zzxsk})
with Minkowski's inequality, we majorize $||R_k f||^\sim_{p,\nu}$ by the following:
\bea
&&\left (\!\sigma_{n-k-1}\!\intl_0^\infty \intl_{O(n)} \intl_{O(n-k)}\!\!\!\!r^{\nu p+n-k-1} | (R_k f)(\gam v_0, r \om e_1)|^p \,d\om d\gam dr\right )^{1/p}\qquad\nonumber\\
&&\le
\sigma_{n-k-1}^{1/p}\intl_{\bbr^k}
 \left (\intl_0^\infty \intl_{O(n)} \intl_{O(n-k)} |f(r\gam \tilde\om (e_1\!+\!z))|^p
r^{\mu p+n-1}  \,d\om d\gam dr\right )^{1/p}\!\!\!dz\nonumber\\
&&= c_{k} \, ||f||_{p,\mu},\nonumber\eea
\[
c_{k}= \left (\frac{\sigma_{n-k-1}}{\sigma_{n-1}}\right )^{1/p}
\intl_{\bbr^k} (1+|z|^2)^{-(\mu +n/p)/2}\, dz.\]
The last integral was computed in the previous section.
In the case $p=\infty$ we  similarly have
$||R_k f||^\sim_{\infty,\nu}\le c_{k}\, ||f||_{\infty,\mu}$  with $1/p=n/p=0$.
The proof of the equality $c_k=||R_k||$  mimics that in Theorem  \ref{o9iu}.
\end{proof}

\bibliographystyle{amsplain}

\end{document}